\documentclass[article,12pt,reqno]{amsart}
\usepackage[english]{babel}
\usepackage[utf8]{inputenc}
\usepackage{amsmath, amssymb, amsthm}
\usepackage{graphics}
\usepackage[]{todonotes}
\usepackage{physics}
\usepackage{centernot}
\usepackage{comment}
\usepackage[cm]{fullpage}

\usepackage[hidelinks]{hyperref}

\usepackage{mathrsfs}

\newtheorem{theorem}{Theorem}

\newtheorem{remark}[theorem]{Remark}

\newtheorem{lemma}[theorem]{Lemma}

\usepackage{cleveref}

\title{A simple proof of exponential decay for the near-critical planar Ising model}

 \author{Jianping Jiang}
\address{Jianping Jiang \\ Yau Mathematical Sciences Center, Tsinghua University, Beijing 100084, China.}
\email{jianpingjiang@tsinghua.edu.cn}

\author{Frederik Ravn Klausen}
\address{Frederik Ravn Klausen \\ University of Cambridge, DPMMS, Cambridge, United Kingdom}
\email{frk23@cam.ac.uk}

\begin{document}
\begin{abstract}
    For the Ising model defined on $a\mathbb{Z}^2$ at critical temperature with external field $a^{15/8}h$, we give a simple and elementary proof that its truncated two-point function decays exponentially. The proof combines the high temperature expansion, random-cluster and random current representations.
    A new input in the proof is that, in the near-critical sourceless single current measure, there are many loops formed by a path on $a\mathbb{Z}^2$ with diameter of order $1$, together with two external edges that connect the path's endpoints to the ghost.
\end{abstract}

\maketitle

\vspace{-0.8cm}
\section{Introduction and main result}
The two-dimensional near-critical Ising model in the presence of a nonzero external field is widely regarded as particularly challenging to study \cite{Mcc12}. Nevertheless, Zamolodchikov proposed a solution directly in the scaling limit and made remarkable predictions about the particle spectrum, including the existence of eight particles whose masses are related to the Lie algebra $E_8$ \cite{Zam89a, Zam89b}. On the mathematical side, the scaling limits of the critical and near-critical magnetization fields were established in \cite{CGN15, CGN16}. Afterwards, exponential decay was proved for this near-critical model in \cite{CJN20a}, implying that the corresponding field theory has a mass gap. The goal of this paper is to present an elementary and less technical proof of the exponential decay result. 
 
For any finite domain $\Lambda \subset \mathbb{R}^2$, we write $\Lambda^a:=a\mathbb{Z}^2 \cap \Lambda$ for its {\it a-approximation}. The Ising model at the inverse critical temperature $\beta_c$ on $\Lambda^a$ with external field $a^{15/8}h$ and free boundary conditions is the probability measure $\mu^f_{\Lambda,a,h}$ on $\{-1,+1\}^{\Lambda^a}$ defined by
\[\mu^f_{\Lambda,a,h}(\sigma) \propto \exp[\beta_c\sum_{uv}\sigma_u\sigma_v+a^{15/8} h\sum_{u\in \Lambda^a}\sigma_u],\]
where the first sum is over all nearest-neighbor edges $uv$ in $\Lambda^a$. Let $\langle \cdot \rangle^f_{\Lambda,a,h}$ be the expectation with respect to $\mu^f_{\Lambda,a,h}$ and $\langle \sigma_x; \sigma_y \rangle^f_{\Lambda,a,h}:=\langle \sigma_x \sigma_y \rangle^f_{\Lambda,a,h}-\langle \sigma_x\rangle^f_{\Lambda,a,h}\langle \sigma_y \rangle^f_{\Lambda,a,h}$ be the covariance. For any $N>0$, let $\Lambda_N:=[-N,N]^2$. For $x,y\in\mathbb{R}^2$, let $|x-y|$ denote the Euclidean distance between $x$ and $y$. Our main result is the following.
\begin{theorem}\label{thm:main}
	There exist $C_0, C_1 \in (0,\infty)$ such that for each $a\in(0,1]$,  $h\in(0,a^{-15/8}]$ and $N>0$,
	\[\langle \sigma_x; \sigma_y \rangle^f_{\Lambda_N,a,h}\leq C_0 a^{1/4}|x-y|^{-1/4}\exp[-C_1 h^{8/15}|x-y|],~\forall x\neq y\in \Lambda_{N/2}^a.\]
\end{theorem}

The infinite-volume limit $\langle \cdot \rangle_{a,h}$ obtained from $\langle \cdot \rangle_{\Lambda, a,h}$ with $a\in(0,1]$ and $h \in (0,a^{-15/8}]$ is the near-critical regime; $\langle \cdot \rangle_{a,h}$ with $h=0$ is the critical regime. By sending $N\rightarrow \infty$, we see that $\langle \sigma_x;\sigma_y \rangle_{a,h}$ also decays exponentially in the near-critical regime. This decay was first proved in \cite{CJN20a} using random-cluster representation together with conformal loop and conformal measure ensembles. Subsequently, an alternative proof was given in \cite{KR22} using random currents and random-cluster representations. As in \cite{CJN20a}, the discrete result in Theorem \ref{thm:main} can be readily transferred to exponential decay in the continuum limit. We expect that the new proof presented here will further advance the rigorous study of the particle spectrum initiated in \cite{camia2020gaussian,camia2021gaussian}. We refer to \cite{camia2021gaussian} for the expected asymptotic behavior of $\langle \sigma_x;\sigma_y \rangle_{a,h}$ and more.

We outline the main ideas of the proof. By the random current representation, see \eqref{eq:truncatedtwopoint} below, $\langle \sigma_x; \sigma_y \rangle^f_{\Lambda,a,h}$ is bounded above by the probability of no connection between $x$ and $\mathfrak{g}$ in the double current measure $\mathbf{P}^{\{x,y\}}_{\Lambda,a,h}\times\mathbf{P}^{\emptyset}_{\Lambda,a,h}$. The source constraint in $\mathbf{P}^{\{x,y\}}_{\Lambda,a,h}$ forces the existence of a path from $x$ to $y$. This path must cross at least $c|x-y|$ disjoint rectangles of size $3\times 6$ in the easy direction (see Figure \ref{fig:proofidea}). A crucial new ingredient is that, under $\mathbf{P}^{\emptyset}_{\Lambda,a,h}$, the number of rectangles that contain a loop crossing the rectangle along its long direction and connecting to $\mathfrak{g}$ behaves like a binomial distribution with parameters $c|x-y|$ and some $p>0$. This in turn yields the desired exponential decay.

\begin{figure}
	\begin{center}
		\includegraphics{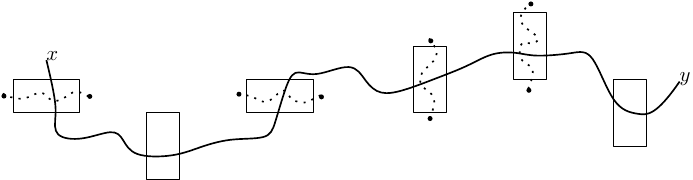}
		\caption{The curve illustrates the path from $x$ to $y$ forced by the source constraint in $\mathbf{P}^{\{x,y\}}_{\Lambda,a,h}$. Each dotted curve represents a loop from $\mathbf{P}^{\emptyset}_{\Lambda,a,h}$ (only the path on $a\mathbb{Z}^2$ is shown and the two marked endpoints of the path are connected to $\mathfrak{g}$ directly). }\label{fig:proofidea}
	\end{center}
\end{figure}

To show that $\mathbf{P}^{\emptyset}_{\Lambda,a,h}$ contains the desired loops through the ghost, we first establish their existence in the random-cluster measure $\phi_{\Lambda,a,h}$ using near-critical RSW and one-arm estimates. A positive fraction of these loops persist in the loop $O(1)$ measure $\ell_{\Lambda,a,h}$ since the uniform even subgraph of $\omega \sim \phi_{\Lambda,a,h}$ (i.e.,the uniform measure over all spanning subgraphs of $\omega$ where each vertex has an even degree) is distributed as $\ell_{\Lambda,a,h}$. Finally, loops in $\ell_{\Lambda,a,h}$ also appear in $\mathbf{P}_{\Lambda,a,h}^{\emptyset}$ due to the stochastic domination $\ell_{\Lambda,a,h} \leq_{s.t.} \mathbf{P}_{\Lambda,a,h}^{\emptyset}$. The main difference of the current proof and the previous approaches in \cite{CJN20a,KR22} is the use of the loop $O(1)$ measure, which transfers the loops from the random-cluster measure into the current measure.


\section{Proof of the main result}
\subsection{Graphical representations of the Ising model}
In this subsection, we briefly introduce three graphical representations and some of their properties that will be used in our proof.

Consider a finite subgraph $G=(V,E)$ of $a\mathbb{Z}^2$. To deal with the external field, it is convenient to add an extra vertex $\mathfrak{g}$ (usually called the {\it ghost} vertex) and add an edge $z\mathfrak{g}$ for each $z\in V$. So the resulting new graph is $\overline{G}=(\overline{V}, \overline{E})$ with $\overline{V}:=V\cup\{\mathfrak{g}\}$ and $\overline{E}:=E\cup \{z\mathfrak{g}: z\in V\}$. The edges in $a\mathbb{Z}^2$ are called {\it internal edges}, while those in $\{z\mathfrak{g}: z\in a\mathbb{Z}^2\}$ are called {\it external edges}.
Recall that the coupling constant $J_{uv}=\beta_c$ for each nearest-neighbor edge $uv$ in $a\mathbb{Z}^2$. We set $J_{z\mathfrak{g}}:=a^{15/8}h$ for each $z\in a\mathbb{Z}^2$. 
\subsubsection*{Random current representation}
For ${\bf n}\in \Omega_{\overline{G}}:=\{0,1,2,\dots\}^{\overline{E}}$, the set of sources of ${\bf n}$ is defined by
\[\partial {\bf n}:=\{x\in \overline{V}: \sum_{y \in \overline{V}: y \sim x} {\bf n}_{xy} \text{ is odd}\}.\]
The random current measure with sources $A\subset \overline{V}$ is the probability measure
\[\mathbf{P}^A_{G,a,h}({\bf n}) \propto w({\bf n}):=\prod_{e \in \overline{E}} \frac{J_e^{{\bf n}_{e}}}{{\bf n}_e !},~ \forall {\bf n}\in \Omega_{\overline{G}} \text{ with }\partial {\bf n} =A,\]
where $J_e=J_{uv}$ for each $e=uv \in \overline{E}$, and ${\bf n}_e \in \{0,1,2,\dots,\}$ is the value of the current at $e$.

One important tool in the study of random current representation is the switching lemma (see Lemma~3.2 of \cite{Aiz82} or Lemma 2.1 of \cite{ABF87}), which yields the following immediately, 
\begin{equation}\label{eq:truncatedtwopoint}
	\langle \sigma_x; \sigma_y \rangle^f_{G, a,h}=\langle \sigma_x \sigma_y \rangle^f_{G,a,h} \mathbf{P}^{xy,\emptyset}_{G,a,h}\left(x \overset{{\bf n} +{\bf m}} {\centernot \longleftrightarrow} \mathfrak{g}\right),
\end{equation}
where $\mathbf{P}^{xy,\emptyset}_{G,a,h}$ is the product of two independent random current measures $\mathbf{P}^{\{x, y\}}_{G,a,h}$ and $\mathbf{P}^{\emptyset}_{G,a,h}$. An edge $e$ in ${\bf n}\in \Omega_{\overline{G}}$ is {\it open} if ${\bf n}_e\geq 1$ and {\it closed} if ${\bf n}_e=0$; $x \overset{{\bf n} +{\bf m}} {\centernot \longleftrightarrow} \mathfrak{g}$ means that there is no path from $x$ to $\mathfrak{g}$ using open edges in ${\bf n}+{\bf m}$ (here $({\bf n}+{\bf m})_e={\bf n}_e+{\bf m}_e$).

\subsubsection*{High-temperature expansion}
The high-temperature expansion of the Ising model (or loop $O(1)$ model) is the probability measure on the set of even subgraphs of $\overline{G}$:
\[\ell_{G,a,h}(F) \propto \prod_{e\in F} \tanh(J_e),~\forall F\subset \overline{E} \text{ with }\partial F=\emptyset,\]
where $\partial F$ is the set of $v\in \overline{V}$ with odd degree in the subgraph $(\overline{V},F)$.

\subsubsection*{The random-cluster model}
Let $\partial G:=\{x\in V: x \text{ has a nearest-neighbor in }a\mathbb{Z}^2\setminus V \}$ and $\partial \overline{G}:=\partial G\cup\{\mathfrak{g}\}$. The {\it boundary conditions} $\xi$ on $\overline{G}$ are given by a partition of $\partial\overline{G}$. Two vertices of $\partial\overline{G}$ are {\it wired} together if they belong to the same element of $\xi$. The random-cluster model on $\overline{G}$ with boundary conditions $\xi$ is the probability measure on $\{0,1\}^{\overline{E}}$ defined by
\[\phi^{\xi}_{G,a,h}(\omega) \propto \prod_{e\in \overline{E}}(1-e^{-2J_e})^{\omega_e}(e^{-2J_e})^{1-\omega_e}2^{\kappa(\omega^{\xi})}.\]
Here, $\omega$ is viewed as the subgraph $(\overline{V},\{e\in \overline{E}: \omega_e=1\})$; $\omega^{\xi}$ denotes the graph obtained from $\omega$ by identifying the wired vertices according to $\xi$; and $\kappa(\omega^{\xi})$ denotes the number of connected components of $\omega^{\xi}$. We use $0$ for the {\it free} boundary conditions (i.e., no wiring at all) and $1$ for the {\it wired} boundary conditions (i.e., all vertices of $\partial \overline{G}$ are wired together).

\subsubsection*{Couplings}  
The Ising measure $\mu^f_{G,a,h}$ and the random-cluster measure $\phi^0_{G,a,h}$ are related by \cite{FK72,ES88},
\begin{equation}\label{eq:ES}
	\langle \sigma_x \sigma_y \rangle_{G,a,h}^f=\phi^0_{G,a,h}(x\longleftrightarrow y).
\end{equation}
 $\ell_{G,a,h}$ and $\phi^0_{G,a,h}$ are coupled as follows \cite{EEV02,GJ09}: If $\omega \sim \phi^0_{G,a,h}$, then the uniform even subgraph of $\omega$ has the distribution of $\ell_{G,a,h}$. 
 
In addition, the open edges of $\mathbf{P}_{G,a,h}^{\emptyset}$ have the distribution given by first sampling $F \sim \ell_{G,a,h}$  and then adding to $F$ each $e\in \overline{E}\setminus F$ with probability $1-\sech(J_e)$ independently of each other  (see Remark 3.4 of \cite{DC18} and \cite{LW16}). For additional couplings and an overview of these constructions, see \cite{HJK25}.

The last ingredients we need are a near-critical RSW estimate and a one-arm estimate for $\phi^1_{G,a,h}$. Let $T:=[0,10]\times [0,3]$ and $S:=[1,9]\times [1,2]$. The graph $a\mathbb{Z}^2$ is planar, so we can define its dual edges; we refer to \cite{Gri06} for the exact definition but only mention that a dual edge is open if and only if the corresponding primal edge is closed. Section 8.3 of \cite{DCM22} is about the scaling relation for the random-cluster model with external field, which is independent of the rest of that paper. Lemma 8.5 from that section implies that there exist $c_0,C_1,C_2,h_0\in(0,\infty)$ such that for all $a\in(0,1]$ and $h\in [0,h_0]$,
\begin{equation}\label{eq:nearcriticalRSW}
	\phi^1_{T,a,h}(\exists \text{ dual open circuit in }T^a\setminus S^a \text{ surrounding }S^a)>c_0,
\end{equation}
\begin{equation}\label{eq:onearm}
	C_1(a/r)^{1/8} \leq \phi^0_{\Lambda_{r},a,h}(0 \overset{a\mathbb{Z}^2}{\longleftrightarrow} \partial \Lambda^a_{r}) \leq \phi^1_{\Lambda_{r},a,h}(0 \overset{a\mathbb{Z}^2}{\longleftrightarrow} \partial \Lambda^a_{r}) \leq C_2 (a/r)^{1/8},~\forall r\in [a,1/2], \tag{1-arm}
\end{equation}
where the lower bound of the last inequality is actually due to the stochastic domination $\phi^1_{\Lambda_{r},a,h} \geq _{s.t.} \phi^0_{\Lambda_{r},a,0}$ and Lemma 5.4 of \cite{DCHN11}; see also Lemma 6 and Proposition 7 of \cite{CJN20b} for similar results (with longer proofs).

\subsection{Proof of Theorem \ref{thm:main}}
Let $R$, $S$, $T$, $Q_1$ and $Q_8$ be defined as in Figure \ref{fig:Looptog}. $E(R)$ denotes the event that there is an open path within $T^a$ crossing $R^a$ in the long direction, and that this path is connected to $\mathfrak{g}$ (at least) in each of the two connected components of $T^a\setminus R^a$. $H(R)$ denotes the (random-cluster) event that there exists a dual open circuit in $T^a\setminus S^a$ surrounding $S^a$, and that $S^a$ contains an internal path crossing $S^a$ in the long direction, which is connected to $\mathfrak{g}$ by exactly two external edges whose endpoints (other than $\mathfrak{g}$) lie in $Q^a_1$ and $Q^a_8$, respectively. We emphasize that if $H(R)$ occurs, then exactly two open external edges intersect the cluster of the crossing path in $S^a$.

The following lemma says that the event $H(R)$ occurs with positive probability regardless of boundary conditions. We will prove this lemma after we prove Theorem \ref{thm:main}.
\begin{lemma}\label{lem:FK}
	 There exists $h_0 \in (0,\infty)$ such that, for each $h\in(0,h_0]$, there is $p_h\in(0,\infty)$ satisfying
		\[\phi_{T,a,h}^{\xi}(H(R))\geq p_h,~\forall a\in(0,1],\forall \text{ boundary conditions }\xi.\]
\end{lemma}

The next lemma says that there are many loops passing through the ghost in the single random current.     
\begin{lemma}\label{lem:binomial}
	 There exists $h_0 \in (0,\infty)$ such that, for each $h\in(0,h_0]$, there is $p_h\in (0,\infty)$, such that for any set of $6 \times 3$ rectangles $R_1, \dots, R_n$ in $\Lambda$ of $l^{\infty}$-distance at least $6$ from each other and at least $3$ from $\partial \Lambda$, it holds that $$\mathbf{P}_{\Lambda,a,h}^{\emptyset}[\cap_{i=1}^n[E(R_i)^c] ] \leq \exp[-p_h n/2], \forall a\in(0,1].$$
\end{lemma}
\begin{proof}
The corresponding enlarged rectangles $T_i$ of $R_i$, defined as in Figure \ref{fig:Looptog} with appropriate translations and rotations, all lie in $\Lambda$. Given $\omega$, for each $1\leq i \leq n$ with $\omega \in H(R_i)$,  let $\eta_i=\eta_i(\omega)$ be a loop formed by the two open external edges from $H(R_i)$ together with a path in $T_i$ connecting them. Let 
\[I=I(\omega):=\{i: 1\leq i \leq n, \omega \in H(R_i) \}.\]\
Let $\text{UEG}_{\omega}$ be the uniform even subgraph of $\omega$ and $\eta_0 \sim \text{UEG}_{\omega}$. For any subset $J \subset I$, the even graph $\eta_0 \Delta \eta_J$ with $\eta_J:=\cup_{i\in J}\{\eta_i\}$ is also distributed as $\text{UEG}_{\omega}$ (see, e.g., Corollary 3.3 of \cite{ADCTW19}). So
\[\text{UEG}_{\omega}(\{\eta_0: \eta_0 \in \cap_{i=1}^n[E(R_i)^c]\})=\text{UEG}_{\omega}(\{\eta_0: \eta_0 \Delta \eta_J \in \cap_{i=1}^n[E(R_i)^c]\}),\ \forall J\subset I.\]
For any even graph $\eta \subset \omega$ and $i \in I$, in $\eta$ and $\eta \triangle \eta_i$ the two external edges from $H(R_i)$ to the ghost are either both closed or both open. In the latter case, a parity consideration gives a path between the edges in the plane and thus at least one of $\{\eta, \eta \triangle \eta_i\}$ is in $E(R_i)$ and thus
\begin{align*}
\text{UEG}_{\omega}(\cup_{J \subset I}\{\eta_0: \eta_0 \Delta \eta_J \in \cap_{i=1}^n[E(R_i)^c]\})=\sum_{J \subset I}\text{UEG}_{\omega}(\{\eta_0: \eta_0 \Delta \eta_J \in \cap_{i=1}^n[E(R_i)^c]\}). 
\end{align*}
The last two displayed equations imply that
$$
\operatorname{UEG}_\omega \left(\cap_{i=1}^n[E(R_i)^c] \right)  \leq 2^{ - \abs{I}}.
$$
Integrating over $\phi^0_{\Lambda,a,h}$, the coupling of the loop O(1) and random-cluster model implies that
$$
\ell_{\Lambda,a,h}\left( \cap_{i=1}^n[E(R_i)^c] \right) \leq \phi^0_{\Lambda,a,h}(2^{ - \abs{I}}) \leq \mathbb{E}(2^{ -\text{Bin}(n,p_h)}) = \left(1-\frac{p_h}{2}\right)^n \leq  \exp[-p_h n/2],
$$
where we have used Lemma \ref{lem:FK} in the second inequality.

The same inequality also holds for the measure $\mathbf{P}_{\Lambda,a,h}^{\emptyset}$ because of the stochastic domination $\ell_{\Lambda,a,h} \leq_{s.t.} \mathbf{P}_{\Lambda,a,h}^{\emptyset}$.
\end{proof}

\begin{proof}[Proof of Theorem \ref{thm:main}]
   We claim that there is $h_0\in(0,\infty)$ such that for $h\in(0,h_0]$, there exist $B(h), m(h)\in(0,\infty)$ satisfying that for all $a\in(0,1]$,
	\begin{equation}\label{eq:claim}
		\mathbf{P}^{xy,\emptyset}_{\Lambda_N,a,h}\left(x \overset{{\bf n} +{\bf m}} {\centernot \longleftrightarrow} \mathfrak{g}\right)\leq B(h)\exp[-m(h)|x-y|],~\forall x,y\in \Lambda_{N/2}^a \text{ with }|x-y|>2\sqrt{2},~\forall N>1.
	\end{equation}

    For the proof of the claim, we fix $N>1$ and write $\Lambda=\Lambda_N$. 
	By the source constraint, $\mathbf{P}^{\{x,y\}}_{\Lambda,a,h}$ contains a path (say $\gamma$) from $x$ to $y$ in $\Lambda$. We claim that $\gamma$ must cross at least $n\geq \lfloor |x-y|/72 \rfloor$ rectangles of size $6\times 3$ in the easy direction, and these rectangles are of $l^{\infty}$-distance at least $6$ from each other and at least $3$ from $\partial \Lambda$; this claim will be proved in Lemma \ref{lem:numberofdisjointR} in the appendix. Those rectangles are denoted by $R_1=R_1(\gamma),\dots,R_n=R_n(\gamma)$. 
	
	\begin{figure}
		\begin{center}
			\includegraphics{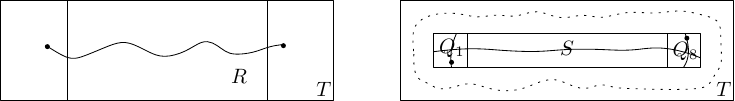}
			\caption{An illustration of the events $E(R)$ (left) and $H(R)$ (right). Here $T:=[0,10]\times[0,3]$, $R:=[2,8]\times [0,3]$, $S:=[1,9]\times[1,2]$, $Q_1:=[1,2)\times [1,2]$ and $Q_8:=(8,9]\times[1,2]$. The dotted circuit is a dual open circuit in $T^a\setminus S^a$. $S^a$ has a left-right crossing. Both $Q^a_1$ and $Q^a_8$ have a top-bottom crossing. In the crossing cluster of $S^a$ only the two marked points in $Q^a_1$ and $Q^a_8$ are connected to the ghost directly by one external edge. Note that $H(R)\subset E(R)$.}\label{fig:Looptog}
		\end{center}
	\end{figure}

    We emphasize that $R_1,\dots,R_n$ are determined by $\mathbf{P}^{\{x,y\}}_{\Lambda,a,h}$, the event $E(R_i)$ can be realized in $\mathbf{P}^{\emptyset}_{\Lambda,a,h}$, which is independent of $\mathbf{P}^{\{x,y\}}_{\Lambda,a,h}$. So Lemma \ref{lem:binomial} implies that
    \[\mathbf{P}^{xy,\emptyset}_{\Lambda,a,h}\left(x \overset{{\bf n} +{\bf m}} {\centernot \longleftrightarrow} \mathfrak{g}\right)\leq \max_{\gamma} \mathbf{P}^{\emptyset}_{\Lambda,a,h}\left(\cap_{i=1}^n E(R_i)^c\right) \leq \exp[-p_h n/2].\]

	The rest of the proof of the theorem is similar to that in \cite{CJN20a}. We claim that it suffices to prove the following: there is $h_0\in(0,1)$ such that for each $h\in(0,h_0]$, there exist $C_3(h),m(h)\in(0,\infty)$ satisfying that for all $a\in(0,1]$,
	\begin{equation}\label{eq:sufficient}
		\langle \sigma_x; \sigma_y \rangle_{\Lambda_N,a,h}^f \leq C_3(h)a^{1/4}|x-y|^{-1/4}\exp[-m(h)|x-y|],~\forall x\neq y\in \Lambda_{N/2}^a,~\forall N>0.
	\end{equation}
	This is because by the GHS inequality \cite{GHS70}, \eqref{eq:sufficient} can be extended from $h\in(0,h_0]$ to $(0,1]$; and then rephrasing \eqref{eq:sufficient} on the subset of $\mathbb{Z}^2$ gives that for all $H\in(0,1]$ and $N>0$,
	\[	\langle \sigma_x; \sigma_y \rangle_{H^{-8/15}\Lambda_N,1,H}^f \leq C_3(1)|x-y|^{-1/4}\exp[-m(1)H^{8/15}|x-y|],~\forall x\neq y\in (H^{-8/15}\Lambda_{N/2}) \cap \mathbb{Z}^2,\]
	which is equivalent to the statement in the theorem.

    The GHS inequality, GKS inequality \cite{Gri67,KS68} and Proposition 5.5 of \cite{DCHN11} imply that
	\begin{equation}\label{eq:shortdistance}
		\langle \sigma_x; \sigma_y \rangle_{\Lambda_N,a,h}^f \leq \langle \sigma_x; \sigma_y \rangle_{\Lambda_N,a,0}^f\leq \langle \sigma_x \sigma_y \rangle_{a\mathbb{Z}^2,a,0}^f\leq C_4 a^{1/4}|x-y|^{-1/4},~\forall x\neq y\in \Lambda_N^a.
	\end{equation}
	 This proves \eqref{eq:sufficient} for $x,y$ with $|x-y|\leq 2\sqrt{2}$.

  For $x,y$ with $|x-y|> 2\sqrt{2}$, let $\vec{h}$ be the external field set to $0$ in $\Lambda_1(x)\cup\Lambda_1(y)$ and otherwise unchanged. Then the GHS inequality and \eqref{eq:truncatedtwopoint} imply that
	\begin{equation}\label{eq:truncatedtwopoint1}
		\langle \sigma_x; \sigma_y \rangle_{\Lambda_N, a,h}^f \leq \langle \sigma_x; \sigma_y \rangle_{\Lambda_N, a,\vec{h}}^f = \langle \sigma_x \sigma_y \rangle_{\Lambda_N,a,\vec{h}}^f \mathbf{P}^{xy,\emptyset}_{\Lambda_N,a,\vec{h}}\left(x \overset{{\bf n} +{\bf m}} {\centernot \longleftrightarrow} \mathfrak{g}\right).
	\end{equation}
	For each $x,y\in \Lambda_N^a \text{ with }|x-y|>2\sqrt{2}$,  \eqref{eq:ES}, the FKG inequality and \eqref{eq:onearm} give
	\begin{align}\label{eq:twopoint}
		\langle \sigma_x \sigma_y \rangle_{\Lambda_N,a,\vec{h}}^f&=\phi^0_{\Lambda_N,a,\vec{h}}(x\longleftrightarrow y)\leq \phi_{\Lambda_{1}(x),a,h}^1(x\overset{a\mathbb{Z}^2}{\longleftrightarrow} \partial \Lambda_{1}^a(x))\phi_{\Lambda_{1}(y),a,h}^1(y\overset{a\mathbb{Z}^2}{\longleftrightarrow} \partial \Lambda_{1}^a(y))
		\leq (C_2a^{1/8})^2.
	\end{align}
	The proof of \eqref{eq:claim} still holds for $\mathbf{P}^{xy,\emptyset}_{\Lambda_N,a,\vec{h}}$ with the only possible change of $n$ to $n-2$ to exclude the two $R_i$ intersecting $\Lambda_1(x)\cup\Lambda_1(y)$.
    Hence, \eqref{eq:sufficient} for $x,y$ with $|x-y|>2\sqrt{2}$ follows from \eqref{eq:truncatedtwopoint1}, \eqref{eq:twopoint} and \eqref{eq:claim}.
\end{proof}

\begin{remark}
    The use of the uniform even subgraph in the proof can be viewed as an instance of the XOR-trick (XOR-ing with a loop through the ghost). Variants of this trick were used in the study of the Lorentz mirror model \cite{kozma2013mirror}, to show macroscopic loops in the loop $O(n)$ model \cite{crawford2025macroscopic} and for proving equality of the regimes of exponential decay for loop $O(1)$ and Ising in $\mathbb{Z}^d, d \geq 3$ \cite{hansen2025uniform}. 
\end{remark}

	\begin{proof}[Proof of Lemma \ref{lem:FK}]
    All crossing events in the proof refer to internal crossings (i.e., using only edges in $a\mathbb{Z}^2$).
        Let
		\begin{equation}\label{eq:E_1E_2}
			\begin{split}
				&E_1:=\{\exists \text{ dual open circuit in }T^a\setminus S^a \text{ surrounding }S^a\},\\
				&E_2:=\{S^a \text{ has LR open crossing}\}\cap \{Q^a_1 \text{ has TB open crossing}\} \cap \{Q^a_8 \text{ has TB open crossing}\},
			\end{split}
		\end{equation}
	where LR stands for left-right and TB stands for top-bottom. On $E_2$, there is a unique internal cluster that crosses $S^a$ in the long direction; denote this cluster by $\mathscr{C}$. We define the cluster sizes
    \begin{equation}\label{eq:Ndef}
        \mathcal{N}:=|\mathscr{C}|\mathbf{1}[{E_2}], \ 
        \mathcal{N}_1:=|\mathscr{C}\cap Q_1^a|\mathbf{1}[{E_2}], \  \mathcal{N}_8:=|\mathscr{C}\cap Q_8^a|\mathbf{1}[{E_2}].
    \end{equation}
    
    Lemma \ref{lem:moments} in the appendix, together with Markov's inequality, implies that for each $\epsilon>0$ there exists $M=M(\epsilon)\in(0,\infty)$ such that
	\begin{equation}\label{eq:Aevent}
		\phi_{T,a,h}^{\xi}(\mathcal{N}>Ma^{-15/8} | E_1) \leq \epsilon,~\forall \text{ boundary conditions }\xi.
	\end{equation}
	Lemma \ref{lem:moments} and the Paley-Zygmund inequality imply that
	\[\phi_{S,a,h}^0(\mathcal{N}_i >\lambda C_5 a^{-15/8})\geq (1-\lambda)^2 C_5^2/C_7,~\forall i\in\{1,8\},~ \forall \lambda\in(0,1).\]
	
	The occurrence of $E_1$ implies that there is an outermost dual open circuit in $T^a\setminus S^a$. By writing $E_1$ as a disjoint union of such outermost dual circuit (which only depends on internal edges outside of this circuit) and the FKG inequality, we have that for any boundary conditions $\xi$,
	\begin{align*}
		&\phi_{T,a,h}^{\xi}(\mathcal{N}_1 >\lambda C_5 a^{-15/8}, \mathcal{N}_8 >\lambda C_5 a^{-15/8}, E_1)
        \geq \phi_{S,a,h}^{0}(\mathcal{N}_1>\lambda C_5 a^{-15/8})\phi_{S,a,h}^{0}(\mathcal{N}_8>\lambda C_5 a^{-15/8})\phi_{T,a,h}^{\xi}(E_1).
	\end{align*}
	Combining the last two displayed inequalities, 
	\begin{equation}\label{eq:Bevent}
		\phi_{T,a,h}^{\xi}(\min\{\mathcal{N}_1, \mathcal{N}_8\}>\lambda C_5 a^{-15/8} \mid E_1)\geq [(1-\lambda)^2 C_5^2/C_7]^2=:c>0.
	\end{equation}

	We now fix any $\lambda_0\in(0,1)$ and choose $M_0>0$ so that $\epsilon_0<c/2$ in \eqref{eq:Aevent}. Then \eqref{eq:Aevent} and \eqref{eq:Bevent} imply that for any boundary conditions $\xi$,
	\begin{equation}\label{eq:clustersize}
		\phi_{T,a,h}^{\xi}(\mathcal{N} \leq M_0a^{-15/8},  \min\{\mathcal{N}_1, \mathcal{N}_8\}>\lambda_0 C_5 a^{-15/8} \mid E_1)>c/2.
	\end{equation}
	
	For $i=1,8$, define $G_i:=\{\exists \text{ exactly one vertex }v\in \mathscr{C}\cap Q^a_i \text{ with }v\mathfrak{g} \text{ open}\}$. Let $p:=1-\exp[-2ha^{15/8}]$. By the finite-energy property of the random-cluster model (see, e.g., Theorem 3.1 of \cite{Gri06}), the conditional probability that an external edge is open, given the states of all other edges, is bounded between $p/2$ and $p$. So
	\[\phi_{T,a,h}^{\xi}(G_i \mid \mathcal{N}_i=n_i, E_1)\geq n_i (p/2)(1-p)^{n_i-1}\geq n_ip(1-n_ip)/2,~i=1,8.\]
	On the other hand, for $n > 0$,
	\[\phi_{T,a,h}^{\xi}(v\mathfrak{g} \text{ is closed for each  }v \in \mathscr{C}\setminus(Q^a_1\cup Q^a_8) \mid \mathcal{N}=n, E_1)\geq (1-p)^n\geq 1-np.\]
	Since $G_1\cap G_8 \cap \{v\mathfrak{g} \text{ is closed for each  }v \in \mathscr{C}\setminus(Q^a_1\cup Q^a_8)\}\cap E_1 \cap E_2 \subset H(R)$, the last two displayed inequalities (which were obtained using the worst-case boundary conditions for each fixed external edge) imply that
	\begin{align*}
		&\phi_{T,a,h}^{\xi}(H(R) \mid \mathcal{N} \leq M_0a^{-15/8},  \min\{\mathcal{N}_1, \mathcal{N}_8\}>\lambda_0 C_5 a^{-15/8}, E_1)\\
		&\qquad \geq [\lambda_0 C_5a^{-15/8}p(1-M_0 a^{-15/8}p)/2]^2 \times[1-M_0a^{-15/8}p]>0 \text{ if } h \text{ is small and positive},
	\end{align*}
	where the last inequality can be shown using the inequality $x\leq 1-e^{-2x}\leq 2x$ if $x\in[0,1/2]$. This, \eqref{eq:clustersize} and \eqref{eq:nearcriticalRSW} complete the proof of the lemma.
	\end{proof}

\section*{Appendix}
In this Appendix, we establish two ancillary lemmas that are somewhat standard, but whose precise forms do not appear to be available in the literature; we therefore include complete proofs. The first lemma concerns the first and second moments of a typical near-critical cluster, and the second is about the number of disjoint rectangles crossed by a path.
\begin{lemma}\label{lem:moments}
	See Figure \ref{fig:Looptog} for the definition of $S$ and $T$, \eqref{eq:E_1E_2} for $E_1$ and $E_2$, and \eqref{eq:Ndef} for $\mathcal{N},\mathcal{N}_1,\mathcal{N}_8$.
	There exists $h_0, C_5, C_6, C_7\in(0,\infty)$ such that for all $a\in(0,1]$ and $h\in[0,h_0]$,
	\begin{align*}
		&C_5 a^{-15/8}\leq \phi_{S,a,h}^0(\mathcal{N}_i)\leq C_6 a^{-15/8},~\phi_{S,a,h}^0(\mathcal{N}_i^2)\leq C_7 a^{-15/4} \text{ for }i=1,8.\\
		&\phi_{T,a,h}^{\xi}(\mathcal{N} \mid E_1)\leq C_6 a^{-15/8},~\forall \text{ boundary conditions }\xi.
	\end{align*}
\end{lemma}
\begin{proof}
    We may assume $a\in[0,1/100]$ since the $a\in(1/100,1]$ case follows from a finite energy argument. To simplify notation, all $\longleftrightarrow$ in the proof denote internal connection (i.e., using only edges in $a\mathbb{Z}^2$).
	Whenever $E_2$ occurs, each $x\in \mathscr{C}\cap Q^a_1$ must satisfy $x\longleftrightarrow \partial \Lambda_{1}^a(x)$. So, together with \eqref{eq:onearm},
	\[\phi_{S,a,h}^0(\mathcal{N}_1)\leq \phi_{S,a,h}^0\left(\sum_{x\in Q_1^a}\mathbf{1}[x\longleftrightarrow \partial \Lambda^a_{1}(x)]\mathbf{1}[E_2]\right)\leq \sum_{x\in Q_1^a}\phi_{S,a,h}^0(x\longleftrightarrow \partial \Lambda^a_{1}(x))\leq C_6 a^{-15/8}.\]
	
	The upper bound for $\phi_{T,a,h}^{\xi}(\mathcal{N} \mid E_1)$ follows in a similar way, using \eqref{eq:nearcriticalRSW}.
	
	To prove the lower bound, we denote by $E_3$ the event that there are enough LR and TB crossings of $Q^a_1$ such that any path in $Q_1^a$ with diameter $\geq 1/4$ hits one of those crossings. The critical RSW tells us that $\phi_{Q_1,a,0}^0(E_3)>0$. Let $D^a:=((1,1)+[1/4,3/4]^2)\cap a\mathbb{Z}^2$. Then by FKG and \eqref{eq:onearm}, 
	\begin{align*}
		\phi_{S,a,h}^0(\mathcal{N}_1) &\geq \phi_{S,a,h}^0(\mathcal{N}_1\mathbf{1}[E_3])\geq \phi_{S,a,h}^0\left(\sum_{x\in D^a}\mathbf{1}[x\longleftrightarrow\partial \Lambda_{1/4}^a(x)]\mathbf{1}[E_3]\mathbf{1}[E_2]\right)\\
		&\geq \phi_{S,a,h}^0(E_2)  \phi_{S,a,h}^0(E_3) \phi_{S,a,h}^0\left(\sum_{x\in D^a}\mathbf{1}[x\longleftrightarrow\partial \Lambda_{1/4}^a(x)]\right) \geq C_5 a^{-15/8}.
	\end{align*}

	
	For the second moment, using $\{x \longleftrightarrow \mathscr{C}, y\longleftrightarrow \mathscr{C}, E_2\}\subset \{x \longleftrightarrow y\}$,

\begin{align*}\phi_{S,a,h}^0(\mathcal{N}_1^2)\leq \sum_{x,y\in Q_1^a} \phi_{S,a,h}^0(x\longleftrightarrow y)\leq |Q_1^a|+\sum_{x\in Q_1^a}\sum_{k=0}^{\lceil -\log_2 a \rceil}\sum_{y\in Q_1^a: 2^{k}a\leq |y-x|<2^{k+1}a} \phi_{S,a,h}^0(x\longleftrightarrow y).
	\end{align*}
	Note that for all $x,y\in Q_1^a$ with $|y-x|\geq 2^{k}a$, FKG and \eqref{eq:onearm} give
	\[\phi_{S,a,h}^0(x\longleftrightarrow y)\leq \phi_{\Lambda_{2^{k-2}a}(x),a,h}^1(x\longleftrightarrow \partial \Lambda_{2^{k-2}a}^a(x))\phi_{\Lambda_{2^{k-2}a}(y),a,h}^1(y\longleftrightarrow \partial \Lambda_{2^{k-2}a}^a(y))\leq[\tilde{C}_2 2^{-(k-2)/8}]^2,\]
	  where we changed $C_2$ to $\tilde{C}_2$ to include the $k\in\{0,1\}$ case. Plugging this into the last displayed inequality, we get
	\[\phi_{S,a,h}^0(\mathcal{N}_1^2)\leq C_7a^{-15/4}.\]
The proofs for the inequalities involving $Q_8$ are the same.	 
\end{proof}

\begin{lemma}\label{lem:numberofdisjointR}
	Fix $L,N>0$. Let $x,y\in \Lambda_{N/2}$, and $\gamma$ be any path in $\Lambda_N$ connecting $x$ to $y$. Then $\gamma$ must cross at least $\lfloor |x-y| /(24L)\rfloor$ rectangles of size $2L\times L$ in the easy direction, and all those rectangles are of $l^{\infty}$-distance at least $2L$ from each other and $l^{\infty}$-distance at least $L$ from $\partial \Lambda_N$.
\end{lemma}
\begin{proof}
	Without loss of generality, we may assume that $N> 12L$ since otherwise $\lfloor |x-y| /(24L)\rfloor=0$. If $\gamma \subset \Lambda_{N-2L}$, then $\gamma$ contains a subpath $\gamma_1\subset \Lambda_{N-2L}$ whose horizontal or vertical displacement is at least $|x-y|/2$. If $\gamma \cap \partial \Lambda_{N-2L} \neq \emptyset$, then $\gamma$ contains a subpath $\gamma_1\subset \Lambda_{N-2L}$ whose horizontal or vertical displacement is at least
    \[N/2-2L > N/3 >|x-y|/6.\]
    We consider the case $\gamma_1\subset \Lambda_{N-2L}$ whose horizontal displacement is at least $|x-y|/6$ (the other cases are handled similarly). Let
    \[Z:=\{z\in \gamma_1: \text{Re}(z)=4kL \text{ for some } k\in \mathbb{Z}\}.\]
    Then $Z$ contains at least $\lfloor |x-y| /(24L)\rfloor$ many points with distinct real parts and we denote a choice of such points by $Z^*$. For each $z\in Z^*$, $\Lambda_L(z)$ is contained in $\Lambda_N$ and of $l^{\infty}$-distance at least $2L$ from $\Lambda_L(z^{\prime})$ for all $z^{\prime} \in Z^*$ with $z^{\prime} \neq z$. For each $z\in Z^*$, $\gamma_1$ contains a subpath in $\Lambda_L(z)$ from $z$ to $\partial \Lambda_L(z)$ and thus at least one of the following four rectangles is crossed by $\gamma_1$ in the easy direction: $z+[-L,0]\times [-L,L]$, $z+[0,L]\times [-L,L]$, $z+[-L,L]\times [-L,0]$ and $z+[-L,L]\times [0,L]$.
\end{proof}

\section*{Acknowledgments}
JJ was supported by National Natural Science Foundation of China (No. 12271284 and No. 12226001). FRK was supported by the Carlsberg Foundation, grant CF24-0466 and the Independent Research Fund Denmark, DFF-6219-00007B. The authors thank the referees for useful comments and suggestions.

\section*{Declarations}
\subsection*{Data availability\nopunct}
Data sharing is not applicable to this article as no datasets were generated or analysed during the current study.
\subsection*{Conflict of interest\nopunct}
On behalf of all authors, the corresponding author states that there is no conflict of interest.

\bibliographystyle{abbrv}
\bibliography{bibliography}

\end{document}